\documentclass[10pt]{amsart}
\usepackage{amsthm,amsmath,amssymb}
\usepackage{graphicx} 
\usepackage{tikz-network}

\newtheorem{thm}{Theorem} 

\newtheorem{lem}[thm]{Lemma}

\theoremstyle{definition}

\newcommand{\dw}{\overrightarrow{\nu}}
\newcommand{\supp}{\mbox{supp}}

\title{The stable maximum nullity of digraphs and $1$-DAGs}
\author{Marina Arav}
 \address{Department of Mathematics and Statistics, Georgia State University, Atlanta, Georgia , USA}
\author{Hein van der Holst}
  \address{Department of Mathematics and Statistics, Georgia State University, Atlanta, Georgia , USA}

\begin{document}

\maketitle

\begin{abstract}
Given a digraph $D=(V,A)$ with vertex-set $V=\{1,\ldots,n\}$ and arc-set $A$, we denote by $Q(D)$ the set of all real $n\times n$ matrices $B=[b_{u,w}]$ with $b_{u,u}\not=0$ for all $u\in V$, $b_{u,w} \not= 0$ if $u\not=w$ and there is an arc from $u$ to $w$, and
 $b_{u,w}=0$ if $u\not=w$ and there is no arc from $u$ to $w$. We say that a matrix $B\in Q(D)$ has the Asymmetric Strong Arnold Property (ASAP) if $X\circ B = 0$, $X^T B = 0$,  and $B X^T = 0$ implies $X=0$. We define 
 the stable maximum nullity, $\dw(D)$, of a digraph $D$ as the largest nullity of any matrix $A\in Q(D)$ that has the ASAP\@. We show that a digraph $D$ has $\dw(D)\leq 1$ if and only $D$ and $\overleftarrow{D}$ a partial $1$-DAGs.
\end{abstract}


\section{Introduction}

Given a digraph $D=(V,A)$ with vertex-set $V=\{1,\ldots,n\}$ and arc-set $A$, we denote by $Q(D)$ the set of all real $n\times n$ matrices $B=[b_{u,w}]$ with 
\begin{enumerate}
\item $b_{u,u}\not=0$ for all $u\in V$,
\item $b_{u,w} \not= 0$ if $u\not=w$ and there is an arc from $u$ to $w$, and
\item $b_{u,w}=0$ if $u\not=w$ and there is no arc from $u$ to $w$.
\end{enumerate}

We say that a matrix $B\in Q(D)$ has the Asymmetric Strong Arnold Property (ASAP) if $X\circ B = 0$, $X^T B = 0$,  and $B X^T = 0$ implies $X=0$. Observe that since the largest rank of a matrix $A\in Q(D)$ equals $n$, there exists a matrix $A\in Q(D)$ that satisfies the ASAP\@.

For a digraph $D$, we define the parameter $\dw(D)$ to be the largest nullity of any $A\in Q(D)$ that satisfies ASAP\@.

If $D$ is a digraph, we denote by $\overleftarrow{D}$ the digraph obtained from $D$ by reversing each arrow. If $A\in Q(D)$, then $A^T\in Q(\overleftarrow{D})$.

\begin{lem}
    If $A\in Q(D)$ has the ASAP, then $A^T\in Q(\overleftarrow{D})$ has the ASAP\@.
\end{lem}
\begin{proof}
    Suppose $A\in Q(D)$ has the ASAP\@. Then, if $X\circ A  = 0$, $X^T A = 0$, and $A X^T = 0$, then $X=0$. Since $X\circ A = 0$ if and only if $X^T\circ A^T$, we obtain that if $Y\circ A^T = 0$, $A^T Y^T = 0$, and $Y^T A^T = 0$, then $Y=0$. Hence $A^T\in Q(\overleftarrow{D})$ has the ASAP.
\end{proof}

Since the nullity of $A$ equals the nullity of $A^T$, we obtain:

\begin{lem}
For any digraph $D$, $\dw(D)=\dw(\overleftarrow{D})$.
\end{lem}

The support of a vector $x=[x_i]\in \mathbb{R}^n$ is the set of all $i$ with $x_i\not=0$. We use the notation $\supp(x)$ for the support of a vector $x$.
Let $D$ be a digraph. Two subsets $R, S$ (in this order) of $V(D)$ are said to \emph{touch} if $R\cap S\not=\emptyset$ or there is an arc connecting a vertex of $R$ to a vertex of $S$. We say that a matrix $A\in Q(D)$ has the Support Property (SP) if the support of every nonzero vector $x$ in the left null space touches the support of every nonzero vector $y$ in the right null space.

\begin{lem}
    Let $D$ be a digraph. If $A\in Q(D)$ has the ASAP, then $A$ has the SP.
\end{lem}
\begin{proof}
We prove this by contraposition. Suppose $A$ does not have the SP. Then there exist a nonzero vector $x$ in the left null space of $A$ and a nonzero vector $y$ in the right null space of $A$ such that $\supp(x)\cap \supp(y)=\emptyset$ and there is no arc connecting a vertex $v$ in $\supp(x)$ to a vertex $w$ in $\supp(y)$. Let $X=[x_{i,j}] = x y^T$. Then $X^T A = 0$, $A X^T = 0$, $X\circ A = 0$, but $X\not=0$. Hence $A$ does not have the ASAP\@.
\end{proof}

\section{Minors and bipartite graphs}

If $D$ is a digraph, we can construct a bipartite graph $G$ with a perfect matching $M$ as follows. For each vertex $v$ of $D$, there correspond two vertices $v^-$ and $v^+$ and an edge $v^-v^+$ in $G$. If $vw$ is an arc from $v$ to $w$, then in $G$ there is an edge $v^-$ to $w^+$. The matching $M$ consists of all edges $v^-v^+$, for $v$ a vertex of $D$. 

Conversely, if $G$ is a bipartite graph with a perfect matching $M$, we can construct a digraph $D=D(G, M)$ as follows. Label the vertices in one the parts of the bipartition as $v_1^-,v_2^-,\ldots,v_n^-$ and the vertices in the other part as $v_1^+,v_2^+,\ldots,v_n^+$ such that $v_i^-v_i^+$ are the edges in the matching $M$. Let the vertex-set of $D$ be $V := \{v_1,v_2,\ldots,v_n\}$. The arc-set of $D$ consists of all arcs $v_iv_j$ such that $v_i^-v_j^+ \in E(G)-M$. 

In \cite{MR4658905}, Arav et al. introduced the graph parameter $Mb_S(G)$ for each bipartite graph $G$ with a perfect matching.

\begin{thm}
If $G$ is a bipartite graph and $M$ is a perfect matching of $G$, then $\dw(D(G, M)) = Mb_S(G)$.
\end{thm}

From Theorem~10 and~14 in Arav et al.~\cite{MR4658905} we obtain the following lemma.

\begin{lem}\label{lem:subdigraph}
    Let $H$ be a subdigraph of $D$. Then $\dw(H)\leq \dw(D)$.
\end{lem}

An arc in a digraph is \emph{butterfly contractible} if it is the only outgoing or the only incoming one of its endpoints. A digraph $H$ is a \emph{butterfly minor}
of a digraph $D$, if $H$ can be obtained by applying butterfly contractions on a subdigraph of $D$. 

A \emph{bi-directed edge} connecting the vertices $u$ and $w$ in a digraph $D$ is a pair of arcs, one directed from $u$ to $w$ and one directed from $w$ to $u$. \emph{Contracting} a bi-directed edge connecting vertices $u$ and $w$ in a digraph means deleting the arcs between $u$ and $w$, and identifying $u$ and $w$.
In this paper, we extend the notion of butterfly minor by allowing contractions of bi-directed arcs, and call the resulting digraph a \emph{directed minor}.

Let $G$ be a bipartite graph and let $v\in V(G)$ be a vertex with $N_G(v) = \{v_1,v_2\}$ and, for $i=1,2$, let $e_i$ be the edge between $v$ and $v_i$. Let $G'$ be the graph obtained from $G$ by deleting $v$ and identifying the vertices $v_1$ and $v_2$. We say that $G'$ is obtained from $G$ by \emph{bicontracting} $v$; that is, $G'$ is obtained from $G$ by contracting both edges incident with $v$. During this process, multiple edges may appear, but as $G$ is a bipartite graph, then no loops will appear. Observe that $G'$ has a perfect matching if and only if $G$ has a perfect matching, and that  a butterfly contraction of an arc in a digraph $D=D(G, M)$ corresponds to bicontraction of a vertex with degree two in the bipartite graph $G$. 

The following lemma follows directly from Theorem~15 in Arav et al.~\cite{MR4658905}.
\begin{lem}
Let $G$ be a bipartite graph with a perfect matching $M$. Let $G'$ be a bipartite graph obtained from a bipartite graph $G$ by bicontracting a vertex $v$ of degree two. Then $Mb_{S}(G') \leq Mb_{S}(G)$.
\end{lem}

Hence we obtain:

\begin{lem}\label{lem:butterfly}
    Let $H$ be obtained from $D$ by a butterfly contraction. Then $\dw(H)\leq \dw(D)$.
\end{lem}

Let $G$ be a graph. We say that a subgraph $H$ of $G$ is \emph{central} if $G\setminus V(H)$ has a perfect matching. 

Let $C$ be a central $4$-cycle in a bipartite graph $G$, with vertices $v_1,v_2,v_3,v_4$ in this order on $C$.  Let $G'$ be obtained from $G$ by removing the edges between $v_1$ and $v_2$, and between $v_3$ and $v_4$, and identifying the vertices $v_1$ and $v_3$, and identifying the vertices $v_2$ and $v_4$. Observe that between the two vertices obtained after identification, there are at least two edges.  We call this operation a $C_4$-contraction. Observe that a $C_4$-contraction $H$ of a bipartite graph $G$ has a perfect matching if and only if $G$ has a perfect matching. Observe that a contraction of a bi-directed edge in a digraph $D=D(G, M)$ corresponds to a $C_4$-contraction and deleting one of the edge between the two vertices after identifying $v_1$ and $v_3$, and $v_2$ and $v_4$, in the bipartite graph $G$. 

From Theorem~16 in Arav et al.~\cite{MR4658905} we obtain the following lemma.
\begin{lem}
Let $G$ be a bipartite with a perfect matching $M$. Let $G'$ be a bipartite graph obtained from a bipartite graph $G$ by a $C_4$-contraction. Then $Mb_{S}(G') \leq Mb_{S}(G)$.
\end{lem}

Hence we obtain:

\begin{lem}\label{lem:bidirected}
    Let $H$ be obtained from $D$ by contracting a bi-directed edge. Then $\dw(H)\leq \dw(D)$.
\end{lem}

From Lemmas~\ref{lem:subdigraph},~\ref{lem:butterfly}, and~\ref{lem:bidirected}, we obtain the following theorem.

\begin{thm}
    Let $H$ be a directed minor of $D$. Then $\dw(H)\leq \dw(D)$.
\end{thm}

\section{Partial $k$-DAGs}

In \cite{MR2419778}, Hunter and Kreutzer introduced the Kelly-width of a digraph. In this section, we follow their paper.

A \emph{DAG} is an acyclic digraph, that is, a digraph with no directed cycles. Let $T$ be a DAG. If $i$ and $j$ are two distinct nodes of $T$, we write $i\prec j$ if there is a directed walk in $T$ from $i$ to $j$. If $i$ and $j$ are nodes of $T$, we write $i\preceq j$ if either $i=j$ or $i\prec j$.
Let $D$ be a digraph, and let $W, X\subseteq V(D)$. We say that $X$ \emph{guards} $W$ if $W\cap X=\emptyset$, and for all $uv\in E(D)$, if $u\in W$, then $v\in W\cup X$.

A \emph{Kelly-decomposition} of a digraph $D$ is a triple $(T, \mathcal{W}, \mathcal{X})$, where $T$ is a DAG, and $\mathcal{W} = \{W_i\}_{i\in V(T)}$ and $\mathcal{X} = \{X_i\}_{i\in V(T)}$ are families of subsets of $V(D)$ such that
\begin{itemize}
\item $\mathcal{W}$ is a partition of $V(D)$,
\item for all nodes $i\in V(T)$, $X_i$ guards $W_{\preceq i}$, and
\item for each node $i\in V(T)$, the children of $i$ can be enumerated as $j_1,\ldots,j_s$ so that for each $j_q$, $X_{j_q}\subseteq W_i\cup X_i\cup \bigcup_{p<q} W_{\preceq j_p}$. Also, the roots of $T$ can be enumerated as $r_1, r_2,\ldots$ such that for each root $r_q$, $W_{r_q}\subseteq \bigcup_{p<q} W_{\preceq r_p}$.
\end{itemize}
The \emph{width} of a Kelly-decomposition $(T, \mathcal{W}, \mathcal{X})$ is defined as $\max\{|W_i\cup X_i|:i\in V(T)\}$. The \emph{Kelly-width} of $D$ is the minimum width over all possible Kelly-decompositions of $D$.

The class of $k$-DAGs is defined recursively as follows:
\begin{enumerate}
\item A complete digraph with $k$ vertices is a $k$-DAG.
\item A $k$-DAG with $n+1$ vertices can be constructed from a $k$-DAG $H$ with $n$ vertices by adding a vertex $v$ and arcs satisfying:
\begin{itemize}
    \item There are at most $k$ arcs from $v$ to $H$, and
    \item If $X$ is the set of endpoints of the edges added in the previous subcondition, an arc is added from $u\in V(H)$ to $v$ if $uw$ for all $w\in X\setminus \{u\}$. If $X=\emptyset$, then this condition is true by default for all $u\in V(H)$.
\end{itemize}
\end{enumerate}

A \emph{partial $k$-DAG} is a spanning subdigraph of a $k$-DAG.

Let $D=(V,A)$ be a digraph. A \emph{directed elimination ordering} $\lhd$ is a linear ordering on $V$. If $\lhd = (v_1,v_2,\ldots, v_n)$ is a directed elimination ordering of $D$, define 
\begin{itemize}
\item $D_1^\lhd := D$, and 
\item $D_{i+1}^\lhd$ is obtained from $D_i^\lhd$ by deleting $v_i$ and adding new arcs $uv$ if $uv_i, v_iv\in A(D_i^\lhd)$ and $u\not=v$. 
\end{itemize}
The \emph{width} of a directed elimination ordering is the maximum over all $i$ of the out-degree of $v_i$ in $D_i^\lhd$.

\begin{thm}\cite{MR2419778}
Let $D$ be a digraph. Then equivalent are:
\begin{enumerate}
\item $D$ has Kelly-width $\leq k+1$,
\item $D$ has a directed elimination ordering of width $\leq k$, and
\item $D$ is a partial $k$-DAG. 
\end{enumerate}
\end{thm}

Kintali and Zhang \cite{KINTALI201740}  proved the following lemma.
\begin{lem}
    If a digraph $D$ has Kelly-width $\leq m$ and $H$ is a directed minor of $D$, then $H$ has Kelly-width $\leq m$.
\end{lem}

In this paper, we prove that the class of digraphs $D$ with $\dw(D)\leq 1$ is equal to the class of digraphs $D$ such that both $D$ and $\overleftarrow{D}$ have Kelly-width $\leq 2$.

\section{Characterizations}

Following \cite{KINTALI201740}, we denote, for a positive integer $n$, by $K_n$ the digraph on $n$ vertices with between each pair of distinct vertices $i,j$, an arc from $i$ to $j$ and an arc from $j$ to $i$.

\begin{lem}\label{lem:K2directedcycle}\cite{KINTALI201740}    
A digraph $D$ has Kelly-width $\leq 1$ if and only if it contains no $K_2$-minor.
\end{lem} 

\begin{lem}
    $\dw(K_2)=1$.
\end{lem}
\begin{proof}
    Let $A$ be the $2\times 2$ all-ones matrix. Then $A$ has the ASAP\@. Hence $\dw(K_2)>0$. Since, clearly, $\dw(K_2)\leq 1$, we obtain $\dw(K_2)=1$.
\end{proof}

\begin{thm}
    Let $D$ be a digraph. The following are equivalent:
    \begin{enumerate}
        \item $D$ has no directed cycles, 
        \item $D$ has Kelly-width $1$, and
        \item $\dw(D)=0$.
    \end{enumerate}
\end{thm}
\begin{proof}
  It is clear that $D$ has no $K_2$ as a directed minor if and only if $D$ has no directed cycle.

Suppose $\dw(D)>0$. Let $A\in Q(D)$ with nullity $>0$ that has the ASAP\@. Let $x$ be a non-zero vector in the right null space of $A$; let $v_1\in \supp(x)$. Since $a_{v_1,v_1}\not=0$, there exists an arc $v_1v_2$ from $v_1$ to $v_2$ such that $x_{v_2}\not=0$. As $D$ is finite, repeating this with $v_2$ instead of $v_1$ leads to a directed cycle.

    Suppose now that $\dw(D)=0$ and $D$ has a directed cycle. Then $D$ contains a $K_2$ as a directed minor. As $\dw(K_2)=1$, we obtain a contradiction. 
    
By Lemma~\ref{lem:K2directedcycle}, a digraph $D$ has Kelly-width $\leq 1$ if and only if $D$ has no $K_2$-minor.
\end{proof}

Let $K_{2,2}^=$ be the graph obtained from $K_{2,2}$ by adding an edge parallel to each edge.

\begin{lem}\cite{MR4658905}
$Mb_S(K_{2,2}^=) = 2$.
\end{lem}

Let $N_4$ be the digraph with vertex-set $V=\{a,b,c,d\}$ and arc-set $A=\{ab,ba,bc,cb,cd,dc,ac,db\}$, and let $M_5$ be the digraph with vertex-set $V=\{a,b,c,d,e\}$ and arc-set $A=\{ab,ba,bc,cb,cd,dc,de,ed,ac,ec\}$. Here, if $v$ and $w$ are vertices, then $vw$ denotes the arc from $v$ to $w$. See \cite{KINTALI201740} for pictures of these two digraphs.

\begin{lem}\label{lem:kwforbidden}
    $\dw(K_3)=2$, $\dw(N_4) \geq 2$, $\dw(M_5)\geq 2$,
    $\dw(\overleftarrow{N_4}) \geq 2$, $\dw(\overleftarrow{M_5})\geq 2$.
\end{lem}
\begin{proof}
    For $K_3$, take the $3\times 3$ all-ones matrix. Then $A$ has the ASAP\@. Hence $\dw(K_3)>1$. Since $\dw(K_3)\leq 2$, we obtain $\dw(K_3)=2$.

    To see that $\dw(N_4) \geq 2$, let $G$ be the corresponding bipartite graph of $N_4$. Bicontracting the vertices of degree $2$, gives the graph $K_{2,2}^=$. Since $Mb_S(K_{2,2}^=) = 2$ and $K_{2,2}^=$ is a matching minor of $G$, $Mb_S(G)\geq 2$. Hence $\dw(N_4)\geq 2$.

    To see that $\dw(M_5)\geq 2$, let $G$ be the corresponding bipartite graph of $M_5$. Bicontracting the vertices of degree $2$, deleting the single edge incident with the vertex of degree $5$, and then bicontracting the vertex of degree $2$, gives the graph
    $K_{2,2}^=$. Since $Mb_S(K_{2,2}^=) = 2$ and $K_{2,2}^=$ is a matching minor of $G$, $Mb_S(G)\geq 2$. Hence $\dw(M_5)\geq 2$.

    Since $\dw(\overleftarrow{D}) = \dw(D)$ for any digraph $D$, we obtain $\dw(\overleftarrow{N_4}) \geq 2$ and $\dw(\overleftarrow{M_5})\geq 2$
\end{proof}

The following lemma was shown in \cite{KINTALI201740}.

\begin{lem}
    Let $D$ be a digraph such that each vertex has outdegree $\geq 2$. Then $D$ has a directed minor isomorphic to $K_3$, $N_4$, or $M_5$.
\end{lem}

The contrapositive of the previous lemma is:

\begin{lem}
    Let $D$ be a digraph that has no minor isomorphic to $K_3$, $N_4$, and $M_5$. Then there exists a vertex with outdegree $\leq 1$. 
\end{lem}

\begin{lem}
    Let $D$ be a digraph with no minor isomorphic to $K_3$, $N_4$, $M_5$, $\overleftarrow{N_4}$, or $\overleftarrow{M_5}$. Then $D$ has a vertex $v$ with outdegree $\leq 1$ and a vertex $w$ with indegree $\leq 1$.
\end{lem}

Observe that if $v$ is a vertex of $D$ with outdegree $1$, then the outgoing arc $a$ at $v$ is butterfly contractible, and if $w$ is a vertex of $D$ with indegree $1$, then the incoming arc $a$ at $w$ is butterfly contractible. If in the previous lemma $v=w$, then by applying a butterfly contraction at one of the arcs incident with $v$, we can obtain distinct vertices $v$ and $w$, with the outdegree of $v$ $\leq 1$ and the indegree of $w$ $\leq 1$. If $v$ is adjacent to $w$, then by a butterfly contraction on the arc $vw$, we can obtain distinct vertices $v$ and $w$ such that $v$ is not adjacent to $w$ and the outdegree of $v$ $\leq 1$ and the indegree of $w$ $\leq 1$. Hence we obtain the following lemma.

\begin{lem}
    Let $D$ be a digraph with no minor isomorphic to $K_3$, $N_4$, $M_5$, $\overleftarrow{N_4}$, or $\overleftarrow{M_5}$. Then $D$ has a vertex $v$ with outdegree $\leq 1$ and a vertex $w$ with indegree $\leq 1$ such that $v\not=w$ and $v$ is not adjacent to $w$.
\end{lem}

If
$A$ is an $n\times n$ matrix, $R, C\subseteq \{1,\ldots,n\}$, then $A[R,C]$ denotes the submatrix of $A$ lying in rows indexed by $R$ and and columns indexed by $C$,
together with the row and column index sets $R$ and $C$.  By $\overline{R}$, we denote the set $\{1,\ldots,n\}-R$. Several abbreviations
are also  used:
$A[R,R]$ can be denoted by $A[R]$, $A[\{v\}, C]$ can be denoted by
$A[v,C]$, etc. Also, $A(R, S) = A[\overline{R}, \overline{S}]$ and $A(R)=A[\,\overline R\,]$.

Let $B$ be an $n\times n$ matrix. If
$S\subseteq \{1,\ldots,n\}$ such that $B[S]$ is nonsingular,   the 
{\em Schur
complement} of $B[S]$ in $B$ is the matrix
$$B/B[S]=B(S)-B[\overline S,S]B[S]^{-1}B[S,\overline S].
$$ 

For a digraph $D=(V,A)$ with vertex-set $V=\{1,\ldots,n\}$ and arc-set $A$, we denote by $Q^0(D)$ the set of all real $n\times n$ matrices $B=[b_{u,w}]$ with $b_{u,w}=0$ if $u\not=w$ and there is no arc from $u$ to $w$. So $b_{u,u}\in \mathbb{R}$ for all vertices $u$ and $b_{u,w}\in \mathbb{R}$ for all arcs $uw$.


\begin{lem}\label{lem:contract}
    Let $D$ be a digraph, let $u$ be a vertex of outdegree $1$, and let $uv$ be the outgoing arc at $u$. If $A=[a_{i,j}]\in Q^0(D)$ has $a_{u,u}\not=0$ and $a_{u,v}\not=0$, then $A/A[u]\in Q^0(D/uv)$ and the nullity of $A/A[u]$ equals the nullity of $A$. Furthermore, if $A$ has the SP with respect to $D$, then $A/A[u]$ has the SP with respect to $D/uv$. 
\end{lem}
\begin{proof}
The Schur complement $A/A[u]$ has the same nullity as $A$. It is easy to see that $A/A[u]\in Q^0(D)$.

Suppose that $A/A[u]$ does not have the SP with respect to $D/uv$. Then there exists a nonzero vector $x'$ in the left null space of $A/A[u]$ and a nonzero vector $y'$ in the right null space of $A/A[u]$ such that $\supp(x')\cap\supp(y')=\emptyset$ and there is no arc in $D/uv$ from a 
vertex in $\supp(x')$ to a vertex in $\supp(y')$.

Define the vector $x$ by 
\begin{equation}\label{eq:x1}
x_u = {x'}^T A[\overline{u},u]/a_{u,u}
\end{equation}
and $x_i = x'_i$ for $i\not=u$.
Then $x$ is a nonzero vector in the left null space of $A$. Define the vector $y$ by 
\begin{equation}\label{eq:y1}
y_u=-\frac{a_{u,v}}{a_{u,u}} y'_v
\end{equation} and $y_i=y'_i$ for $i\not=u$. Then $y$ is a nonzero vector in the right null space of $A$. 

If $\supp(x)\cap\supp(y)\not=\emptyset$, then $\supp(x)\cap\supp(y)=\{u\}$. Hence $y_u\not=0$, and therefore $y'_v\not=0$. Since $x_u\not=0$, an arc exists in $D$ from vertex $w\not=u$ to $u$ and $x'_w\not=0$. Then in $D/uv$ an arc exists from $w$ to $v$. Hence $\supp(x')$ and $\supp(y')$ touch in $D/uv$. This contradiction shows that $\supp(x)\cap\supp(y)=\emptyset$. If a vertex of $\supp(x)$ is adjacent to a vertex in $\supp(y)$, then either a vertex of $\supp(x')$ is adjacent to $u$ or $u$ is adjacent to a vertex in $\supp(y')$. In both case, we obtain that a vertex in $\supp(x')$ is adjacent to a vertex in $\supp(y')$. This contradiction shows that $A$ does not have the SP.
\end{proof}

\begin{lem}\label{lem:delete}
Let $D$ be a digraph, let $u$ be a vertex of outdegree $1$, and let $uv$ be the outgoing arc at $u$. If $A=[a_{i,j}]\in Q^0(D)$ has $a_{u,u}\not=0$ and $a_{u,v}=0$, then $A(u,u)\in Q^0(D-u)$ and the nullity of $A(u,u)$ equals the nullity of $A$. Furthermore, if $A$ has the SP with respect to $D$, then $A(u,u)$ has the SP with respect to $D-u$. 
\end{lem}
\begin{proof}
Since $A(u,u)=A/A[u]$, we see that the nullity of $A(u,u)$ equals the nullity of $A$. It is clear that $A(u,u)\in Q^0(D-u)$.

Suppose that $A(u,u)$ does not have the SP with respect to $D-u$. Then there exists a nonzero vector $x'$ in the left null space of $A(u,u)$ and a nonzero vector $y'$ in the right null space of $A(u,u)$ such that $\supp(x')\cap\supp(y')=\emptyset$ and there is no arc from a vertex in $\supp(x')$ to a vertex in $\supp(y')$.
Define the vector $x$ by $x_u = {x'}^TA[\overline{u},u]/a_{u,u}$ and $x_i = x'_i$ for $i\not=u$. Then $x$ is a nonzero vector in the left null space of $A$. Define the vector $y$ by $y_u=0$ and $y_i=y'_i$ for $i\not=u$. Then $y$ is a nonzero vector in the right null space of $A$. Then $\supp(x)\cap\supp(y)=\emptyset$ and no vertex in $\supp(x)$ is adjacent to a vertex in $\supp(y)$. This shows that $A$ does not have the SP.
\end{proof}

Let $D$ be a digraph, let $u$ be vertex with outdegree $1$, and let $uv$ be the outgoing arc at $u$. By $D*uv$, we denote the digraph obtained from $D$ by first deleting the arcs $wv$ for all vertices $w$ that are in-neighbors of $v$, and then contracting the arc $uv$. Observe that $D*uv$ is a minor of $D/uv$.

\begin{lem}\label{lem:semicontract}
    Let $D$ be a digraph, let $u$ be a vertex of outdegree $1$, and let $uv$ be the outgoing arc at $u$. If $A=[a_{i,j}]\in Q^0(D)$ has $a_{u,u}=0$ and $a_{u,v}\not=0$, then $A(u,v)\in Q^0(D*uv)$ and the nullity of $A(u,v)$ equals the nullity of $A$. Furthermore, if $A$ has the SP with respect to $D$, then $A(u,v)$ has the SP with respect to $D*uv$. 
\end{lem}
\begin{proof}
Since $A(u,v) = A/A[u,v]$, we see that $A(u,v)$ has nullity equal to the nullity of $A$. It is easy to verify that $A(u,v)\in Q^0(D*uv)$.

Suppose that $A(u,v)$ does not have the SP with respect to $D*uv$. Then there exists a nonzero vector $x'$ in the left null space of $A(u,v)$ and a nonzero vector $y'$ in the right null space of $A(u,v)$ such that $\supp(x')\cap \supp(y')=\emptyset$ and no vertex in $\supp(x')$ is adjacent to a vertex in $\supp(y')$.
Define the vector $x$ by $x_u = {x'}^TA[\overline{u},v]/a_{u,v}$ and $x_i = x'_i$ for $i\not=u$, and the vector $y$ by $y_v=0$ and $y_i=y'_i$ for $i\not=v$. Then $x$ is a nonzero vector in the left null space of $A$ and $y$ is a nonzero vector in the right null space of $A$. Then $\supp(x)\cap \supp(y)=\emptyset$ and no vertex in $\supp(x)$ is adjacent to a vertex in $\supp(y)$. This shows that $A$ does not have the SP.
\end{proof}

\begin{lem}\label{lem:noSP}
    Let $D$ be a digraph with $n$ vertices, let $u$ be a vertex of outdegree $1$, let $uw$ be the outgoing arc at $u$, let $v$ be a vertex of indegree $1$, and let $zv$ be the incoming arc. Let $A=[a_{i,j}]\in Q^0(D)$ has $a_{u,u}=a_{u,w}=a_{v,v}=a_{z,v}=0$. If $u\not=v$ and $u$ is not adjacent to $v$, then $A$ does not have the SP. 
\end{lem}
\begin{proof}
Let $x=[x_i]\in \mathbb{R}^n$ with $x_u=1$ and $x_i=0$ for $i\not=u$ and let $y=[y_i]\in \mathbb{R}^n$ with $y_v=1$ and $y_i=0$ for $i\not=v$. Then $x$ belongs to the left null space of $A$ and $y$ belongs to the right null space of $A$. 
Since $u\not=v$, $\supp(x)\cap \supp(y)=\emptyset$. Since $u$ is not adjacent to $v$, $\supp(x)$ and $\supp(y)$ do not touch. Hence $A$ does not have the SP.
\end{proof}

\begin{lem}\label{lem:contradiction}
    Let $D$ be a digraph with no minor isomorphic to $K_3$, $N_4$, $M_5$, $\overleftarrow{N_4}$, or $\overleftarrow{M_5}$. If $A=[a_{i,j}]\in Q^0(D)$, then either $A$ has nullity $\leq 1$ or $A$ does not have the SP.  
\end{lem}
\begin{proof}
    Suppose the statement the lemma is false. Let $D$ be a counterexample with a minimum number of vertices. Then there exists a matrix $A=[a_{i,j}]\in Q^0(D)$ such that $A$ has nullity $>1$ and $A$ has the SP. Since $D$ has no minor isomorphic to $K_3$, $N_4$, $M_5$, $\overleftarrow{N_4}$, or $\overleftarrow{M_5}$, we have that $D$ has a vertex $u$ with outdegree $1$ and a vertex $v$ with indegree $1$ such that $u\not=v$ and $u$ is not adjacent to $v$.
    Let $uw$ be the outgoing arc at $u$ and let $zv$ be the incoming arc at $v$.  If $a_{u,u}=a_{u,w}=a_{v,v}=a_{z,v}=0$, then, by Lemma~\ref{lem:noSP}, $A$ does not have the SP. 
    We may therefore assume that $a_{u,u}\not=0$, $a_{u,w}\not=0$, $a_{v,v}\not=0$, or $a_{z,v}\not=0$. Then Lemmas~\ref{lem:contract},~\ref{lem:delete}, and ~\ref{lem:semicontract} show that there exists a digraph $D'$ with fewer vertices and a matrix $A'\in Q^0(D')$ that has nullity $>1$ and has the SP; a contradiction. Hence the statement of the lemma is true.
\end{proof}

\begin{lem}
    A digraph $D$ has $\dw(D)\leq 1$ if and only if $D$ has no minor isomorphic to $K_3$, $N_4$, $M_5$, $\overleftarrow{N_4}$, or $\overleftarrow{M_5}$.
\end{lem}
\begin{proof}
    Suppose $D$ is a digraph with $\dw(D)\leq 1$. Then, by Lemma, $D$ has no minor isomorphic to $K_3$, $N_4$, $M_5$, $\overleftarrow{N_4}$, or $\overleftarrow{M_5}$.

    Conversely, suppose $D$ has no minor isomorphic to $K_3$, $N_4$, $M_5$, $\overleftarrow{N_4}$, or $\overleftarrow{M_5}$. 
    Suppose for a contradiction that $\dw(D)>1$. Then there exists a matrix $A\in Q(D)$ with nullity $>1$ and having the ASAP\@. Since $Q(D)\subseteq Q^0(D)$, we obtain a contradiction with Lemma~\ref{lem:contradiction}.
\end{proof}

\begin{thm}\label{thm:minorchar}\cite{KINTALI201740}
    A digraph $D$ is a partial $1$-DAG if and only if $D$ does not have a $K_3$, $N_4$, and $M_5$ as a directed minor.
\end{thm}

\begin{thm}
    Let $D$ be a digraph. Then equivalent are:
    \begin{enumerate}
        \item $\dw(D)\leq 1$,
        \item $D$ has no minor isomorphic to $K_3$, $N_4$, $M_5$, $\overleftarrow{N_4}$, or $\overleftarrow{M_5}$, 
        \item $D$ and $\overleftarrow{D}$ are partial $1$-DAGs, and
        \item $D$ and $\overleftarrow{D}$ both have Kelly-width $\leq 2$.  
         \end{enumerate}
\end{thm}

In \cite{MEISTER2010741}, Meister et al. introduced a linear-time algorithm for recognizing digraphs $D$ with Kelly-width $\leq 2$. Applying their algorithm twice, on $D$ and on $\overleftarrow{D}$, gives a linear-time algorithm for recognizing digraphs $D$ with $\nu(D)\leq 1$.

\bibliographystyle{plain}
\bibliography{./biblio}

\end{document}